\documentclass[a4, 12pt]{amsart}
\usepackage[mathscr]{eucal}
\usepackage{amssymb}
\usepackage{latexsym}
\usepackage{amsthm}
\theoremstyle{plain}
\newtheorem{theorem}{Theorem}[section]

\newtheorem{remark}{Remark}[section]

\newtheorem{proposition}{Proposition}[section]

\makeatletter
\@addtoreset{equation}{section}

\title[Eigenvalues of $\mathfrak L$ operator on Self-Shrinkers]
{Estimates for eigenvalues of $\mathfrak L$ operator \\ on Self-Shrinkers*}
\author {Qing-Ming Cheng and Yejuan Peng}
\address{Qing-Ming Cheng \\  \newline \indent Department of Applied Mathematics, Faculty of Sciences, \newline \indent Fukuoka University, Fukuoka  814-0180, Japan. cheng@fukuoka-u.ac.jp}
\address{Yejuan Peng \\  \newline \indent Department of Mathematics,  Graduate School of Science and Engineering \\  \newline \indent Saga University, Saga 840-8502,  Japan. yejuan666@gmail.com}

\begin{document}
\maketitle

\begin{abstract}
In this paper, we study eigenvalues of the closed eigenvalue problem of 
the differential operator $\mathfrak L$, which is introduced
by Colding and Minicozzi in \cite{CM1}, on an $n$-dimensional
compact self-shrinker in $\mathbf {R}^{n+p}$. Estimates for eigenvalues of  
the differential operator $\mathfrak L$ are obtained.
Our estimates for eigenvalues of the differential operator $\mathfrak L$ are sharp.
Furthermore, we also study  the Dirichlet eigenvalue problem of the differential operator $\mathfrak L$
on a bounded domain with a piecewise smooth boundary in an $n$-dimensional complete self-shrinker 
in $\mathbf {R}^{n+p}$.
For Euclidean space $\mathbf {R}^{n}$,  the differential operator $\mathfrak L$ 
becomes  the Ornstein-Uhlenbeck operator in stochastic analysis. Hence, we also give estimates
for eigenvalues of the Ornstein-Uhlenbeck operator.
\end{abstract}

\footnotetext{* Research partially Supported by JSPS Grant-in-Aid for
Scientific Research (B) No. 24340013.}

\section{introduction} 
\noindent
Let $X: M^n\to \mathbf{R}^{n+p}$ be an isometric immersion from  an $n$-dimensional  Riemannian 
manifold $M^n$ into a Euclidean space $\mathbf{R}^{n+p}$.  One considers a smooth one-parameter
family of immersions:
$$
F(\cdot, t):M^n\to  \mathbf{R}^{n+p}
$$
satisfying $F(\cdot, 0)=X(\cdot)$ and 
\begin{equation}
\bigl(\dfrac{\partial F(p,t)}{\partial t}\bigl)^N=H(p,t), \quad (p,t)\in M\times [0,T),
\end{equation}
where $H(p,t)$ denotes the mean curvature vector  of submanifold $M_t=F(M^n,t)$ at point $F(p,t)$.
The equation (1.1) is called the mean curvature flow equation.
A submanifold $X: M^n\to \mathbf{R}^{n+p}$ is said to be a self-shrinker in $\mathbf{R}^{n+p}$ if it satisfies
\begin{equation}
H=-X^N,
\end{equation}
where $X^N$ denotes the orthogonal projection into the normal bundle of $M^n$ (cf. Ecker-Huisken \cite{EH}).
Self-shrinkers play an important role in the study of the mean curvature flow since they are not only 
solutions of the mean curvature flow equation, but they also describe all possible blow ups at a given
singularity of a mean curvature flow. Huisken \cite{H1} proved that the sphere of radius $\sqrt n$ is the only closed
embedded self-shrinker hypersurfaces with non-zero mean curvature. For classifications of complete non-compact 
embedded self-shrinker hypersurfaces,
Huisken \cite{H2} and Colding and Minicozzi \cite{CM1}  proved that an $n$-dimensional  complete 
embedded self-shrinker hypersurface with non-negative mean curvature and polynomial volume growth in  $\mathbf{R}^{n+1}$ is  a Riemannian product $S^k\times \mathbf{R}^{n-k}$,  $0\leq k < n$. Smoczyk \cite{S1} has obtained several results for complete self-shrinkers with higher co-dimensions.\par
\noindent
For study of the rigidity problem for self-shrinkers,   Le and Sesum \cite{LS} and Cao and Li \cite{CL}
have classified   $n$-dimensional  complete embedded self-shrinkers in $\mathbf{R}^{n+p}$ with  polynomial volume growth if the squared norm $|A|^2$ of the second fundamental form satisfies  $|A|^2\leq 1$. For a further study,
see  Colding and Minicozzi \cite{CM2, CM3}, Ding and Wang \cite{DW}, Ding and Xin \cite{DX1, DX2}, Wang \cite{W}  and so on.
\vskip2mm
\noindent
In \cite{CM1}, Colding and Minicozzi introduced a differential operator   $\mathfrak L$ and used it to 
study self-shrinkers. The differential operator $\mathfrak L$ is defined by
\begin{equation}
\mathfrak L f=\Delta f -\langle X,\nabla f\rangle
\end{equation}
for a smooth function $f$, where $\Delta$ and $\nabla$ denote the Laplacian and the gradient 
operator on the self-shrinker, respectively and $\langle \cdot, \cdot \rangle$ denotes the standard
inner product of $\mathbf R^{n+p}$.  We should notice that the differential operator $\mathfrak L$
plays a very important role in studying of $n$-dimensional  complete embedded self-shrinkers in 
$\mathbf{R}^{n+p}$ with  polynomial volume growth  in order to guarantee integration by part holds
as in \cite{CM1}.
\vskip2mm
\noindent
The purpose of  this paper is to  study eigenvalues of the closed eigenvalue problem for 
 the differential operator $\mathfrak L$ on compact 
self-shrinkers  in $\mathbf{R}^{n+p}$ in sections 3 and 4 and  eigenvalues of the Dirichlet
eigenvalue problem of  the differential operator $\mathfrak L$ on a bounded domain with a piecewise smooth boundary  in complete
self-shrinkers  in $\mathbf{R}^{n+p}$ in section 5.  I shall adapt the idea of Cheng and Yang in \cite{CY1} for studying eigenvalues of the Dirichlet
eigenvalue problem of  the Laplacian $\Delta $  to  the differential operator $\mathfrak L$ by constructing appropriated trial functions
for the differential operator $\mathfrak L$.  
Since the differential operator $\mathfrak L$ is self-adjoint
with respect to  measure $e^{-\frac{|X|^2}{2}}dv$, where $dv$ is the volume element of $M^n$
and $|X|^2=\langle X, X\rangle$,
we know that the closed eigenvalue problem:
\begin{equation}
\mathfrak L u=- \lambda u \quad \text{on} \ M^n
\end{equation}
for the differential operator $\mathfrak L$ on compact 
self-shrinkers  in $\mathbf{R}^{n+p}$ has  a real and  discrete spectrum: 
\begin{equation*}
0=\lambda_0<\lambda _{1} \leq  \lambda _{2} \leq \cdots \leq \lambda _{k} \leq \cdots 
 \longrightarrow \infty,
\end{equation*} 
where each eigenvalue is repeated according to  its multiplicity. 
We shall prove the following:

\vskip 2mm
\noindent
\begin{theorem} Let $M^n$ be an $n$-dimensional compact 
self-shrinker  in $\mathbf{R}^{n+p}$. Then, 
eigenvalues of the closed eigenvalue problem ${\rm (1.4)}$
satisfy
\begin{equation}
\sum_{i=0}^k (\lambda_{k+1}-\lambda_{i})^2 \leq \frac{4}{n}\sum_{i=0}^k
(\lambda_{k+1}-\lambda_{i})(\lambda_{i}+\frac{2n-\min_{M^n}{|X|^2}}{4}).
\end{equation} 
\end{theorem}

\begin{remark}
The sphere $S^n(\sqrt n)$ of radius $\sqrt n$ is a compact self-shrinker in $\mathbf{R}^{n+p}$.
For  $S^n(\sqrt n)$ and  for any $k$, the inequality ${\rm (1.5)}$ for eigenvalues 
of the closed eigenvalue problem ${\rm (1.4)}$ becomes  equality.
Hence our results in theorem {\rm 1.1} are sharp. 
\end{remark}

\noindent
Furthermore, from the recursion formula of Cheng and Yang \cite{CY}, we can obtain an upper bound for eigenvalue
$\lambda_{k}$:
\begin{theorem} Let $M^n$ be an $n$-dimensional compact 
self-shrinker  in $\mathbf{R}^{n+p}$. Then, 
eigenvalues of the closed eigenvalue problem ${\rm (1.4)}$
satisfy, for any $k\geq 1$, 
\begin{equation*}
\lambda_{k}+\frac{2n-\min_{M^n}{|X|^2}}{4}
\leq (1+\frac {a(min\{n,k-1\})}{n})(\frac{2n-\min_{M^n}{|X|^2}}{4})k^{2/n},
\end{equation*}
where the bound of $a(m)$ can be formulated as:
$$
\left \{ \aligned
a(0)&\leq 4, \\
 a(1)&\leq 2.64,\\
 a(m)&\leq 2.2-4\log(1+\frac1{50}(m-3)),\qquad \mbox{for}\quad
 m\geq 2.
\endaligned \right .
$$
In particular, for $n\ge 41$ and $k\geq 41$, we have
 \begin{equation*}
\lambda_{k}+\frac{2n-\min_{M^n}{|X|^2}}{4}
\leq (\frac{2n-\min_{M^n}{|X|^2}}{4})k^{2/n}.
\end{equation*}
\end{theorem}
\vskip 1mm
\noindent
Results for eigenvalues of the Dirichlet eigenvalue problem of the differential operator $\mathfrak L$
are given in section 5.

\vskip 2mm
\noindent
{\bf Acknowledgements.}  We would like to express our gratitude to 
the referee for valuable  suggestions and comments.

\section{Preliminaries}
\noindent
Suppose $X:M^{n}\longrightarrow \mathbf{R}^{n+p}$ is an isometric
immersion from Riemannian manifold $M^n$ into the (n+p)-dimensional 
Euclidean space $\mathbf {R}^{n+p}$.  Let $\{E_A\}_{A=1}^{n+p}$ be the standard
basis of $\mathbf{R}^{n+p}$.  The position vector can be written by
$X=(x_1, x_2, \cdots, x_{n+p})$.  
We choose a local orthonormal frame field 
$\{e_1, e_2, \cdots,e_n, e_{n+1}, \cdots, e_{n+p} \}$ and the dual
coframe field 
$\{\omega_1, \omega_2, \cdots,\omega_{n}, \omega_{n+1}, \cdots, \omega_{n+p}\}$
along $M^n$ of  $\mathbf{R}^{n+p}$ such that  $\{e_1, e_2, \cdots,e_n \}$
is a local orthonormal basis on $M^n$. Thus,  we have
$$
\omega_{\alpha}=0, \quad n+1\leq\alpha\leq n+p
$$
on $M^n$.
From the  Cartan's lemma, we have 
$$
\omega_{i\alpha}=\sum_{j=1}^nh^{\alpha}_{ij}\omega_j, \quad  h^{\alpha}_{ij}=h^{\alpha}_{ji}.
$$
The  second fundamental form ${\bf h}$ of $M^n$ and the mean curvature vector $H$
are defined, respectively,  by 
$$
{\bf h}=\sum_{\alpha=n+1}^{n+p}\sum_{i,j=1}^n h^{\alpha}_{ij}\omega_i\otimes \omega_je_{\alpha}
\quad H=\sum_{\alpha=n+1}^{n+p}\sum_{i=1}^n h^{\alpha}_{ii}e_{\alpha}.
$$
One considers the mean curvature flow for a submanifold  $X: M^n\to \mathbf{R}^{n+p}$. Namely,  
we consider a  one-parameter family of immersions:
\begin{equation*}
F(\cdot, t):M^n\to  \mathbf{R}^{n+p}
\end{equation*}
satisfying $F(\cdot, 0)=X(\cdot)$ and 
\begin{equation}
\bigl(\dfrac{\partial F(p,t)}{\partial t}\bigl)^N=H(p,t), \quad (p,t)\in M\times [0,T),
\end{equation}
where $H(p,t)$ denotes the mean curvature vector  of submanifold $M_t=F(M^n,t)$ at point $F(p,t)$.
An important class of solutions to the mean curvature flow equation (2.1) are self-similar shrinkers,
which profiles, self-shrinkers, satisfy  
\begin{equation*}
H=-X^N,
\end{equation*}
which is a system of quasi-linear elliptic partial differential equations of the second order.
Here $X^N$ denotes the orthogonal projection of $X$  into the normal bundle of $M^n$.

In \cite{CM1}, Colding and Minicozzi introduced a differential operator   $\mathfrak L$ and used it to 
study self-shrinkers. The differential operator $\mathfrak L$ is defined by
\begin{equation}
\mathfrak L f=\Delta f -\langle X,\nabla f\rangle
\end{equation}
for a smooth function $f$, where $\Delta$ and $\nabla$ denote the Laplacian and the gradient 
operator on the self-shrinker, respectively.
For a compact self-shrinker $M^n$ without boundary,  we have 
\begin{equation*}
\begin{aligned}
&\int_{M^n}f\mathfrak L u \ e^{-\frac{|X|^2}{2}}dv\\
&=\int_{M^n}f(\Delta u -\langle X,\nabla u\rangle)\ e^{-\frac{|X|^2}{2}}dv\\
&=\int_{M^n}f\text{div}( e^{-\frac{|X|^2}{2}}\nabla u)dv\\
&=\int_{M^n}u \mathfrak L f \ e^{-\frac{|X|^2}{2}}dv,\\
\end{aligned}
\end{equation*}
that is, 
\begin{equation}
\begin{aligned}
&\int_{M^n}f\mathfrak L u \ e^{-\frac{|X|^2}{2}}dv
=\int_{M^n}u \mathfrak L f \ e^{-\frac{|X|^2}{2}}dv,\\
\end{aligned}
\end{equation}
for any smooth functions $u, f$.
Hence, the differential operator $\mathfrak L$ is self-adjoint with respect to the  measure $e^{-\frac{|X|^2}{2}}dv$.
Therefore, we know that the closed eigenvalue problem:
\begin{equation}
\mathfrak L u=- \lambda u \quad \text{on} \ M^n
\end{equation}
has  a real and  discrete spectrum: 
\begin{equation*}
0=\lambda_0<\lambda _{1} \leq  \lambda _{2} \leq \cdots \leq \lambda _{k} \leq \cdots 
 \longrightarrow \infty.
\end{equation*} 
Furthermore,  we have
\begin{equation}
\mathfrak L x_A=-x_A.
\end{equation}
In fact, 
\begin{equation*}
\begin{aligned}
&\mathfrak L x_A=\Delta \langle X, E_A\rangle-\langle X,\nabla x_A\rangle\\
& =\langle \Delta X, E_A\rangle-\langle X,E_A^T\rangle\\
& =\langle H, E_A\rangle-\langle X,E_A^T\rangle\\
& =-\langle X^N, E_A\rangle-\langle X,E_A^T\rangle=-x_A.
\end{aligned}
\end{equation*}
Denote the induced metric by $g$ and define $\nabla u\cdot\nabla v=g(\nabla u,\nabla v)$ for 
functions $u, v$.
We get,  from (2.5), 
\begin{equation}
\mathfrak L |X|^2=\sum_{A=1}^{n+p}\bigl(2x_A\mathfrak L x_A+2\nabla x_A\cdot\nabla x_A\bigl)=2(n-|X|^2).
\end{equation}
Here we have used 
\begin{equation*}
\sum_{A=1}^{n+p}\nabla x_A\cdot\nabla x_A=n.
\end{equation*}
\begin{proposition} For an $n$-dimensional compact  self-shrinker $M^n$  without boundary in $\mathbf{R}^{n+p}$,
we have 
\begin{equation*}
 \min_{M^n}|X|^2\leq n=\dfrac{\int_{M^n}|X|^2 \ e^{-\frac{|X|^2}{2}}dv}
{\int_{M^n} e^{-\frac{|X|^2}{2}}dv}\leq \max_{M^n}|X^N|^2.
\end{equation*}
\end{proposition}
\begin{proof}
Since $\mathfrak L$ is self-adjoint with respect to the  measure $e^{-\frac{|X|^2}{2}}dv$, from (2.6), we have
$$
n\int_{M^n} e^{-\frac{|X|^2}{2}}dv=\int_{M^n}|X|^2 \ e^{-\frac{|X|^2}{2}}dv
\geq \min_{M^n}|X|^2\int_{M^n} e^{-\frac{|X|^2}{2}}dv.
$$
Furthermore, since
\begin{equation}
\Delta |X|^2=2(n+\langle X, H\rangle)=2(n-|X^N|^2),
\end{equation}
we have
$$
n\leq \max_{M^n}|X^N|^2.
$$
It completes the proof of this proposition.
\end{proof}

\section{Universal estimates for eigenvalues}
\noindent
In this section, we give proof of the theorem 1.1.
In order to prove our theorem 1.1, we need to construct trial functions.
Thank to $\mathfrak L X=-X$. We can use  coordinate functions of the position vector $X$ 
of the  self-shrinker  $M^n$ to construct trial functions.

\vskip2mm
\noindent
{\it Proof of Theorem {\rm1.1}}.
For an $n$-dimensional compact  self-shrinker $M^n$  in $\mathbf{R}^{n+p}$,
the closed eigenvalue problem:
\begin{equation}
\mathfrak L u=- \lambda u \quad \text{on} \ M^n
\end{equation}
for the differential operator $\mathfrak L$ has a discrete spectrum.
For any integer $j\geq 0$, let $u_j$ be an eigenfunction corresponding to the eigenvalue $\lambda_j$
such that
\begin{equation}
\begin{cases}
\begin{aligned}
&\mathfrak L u_j=- \lambda_j u_j \quad \text{on} \ M^n\\
&\int_{M^n}u_iu_j \ e^{-\frac{|X|^2}{2}}dv=\delta_{ij}, \ \text{for any} \ i, j. 
\end{aligned}
\end{cases}
\end{equation}
From the Rayleigh-Ritz inequality, we have
\begin{equation}
\lambda_{k+1} \leq \frac{\displaystyle{-\int_{M^n}}\varphi\mathfrak L \varphi  \ e^{-\frac{|X|^2}{2}}dv}
{\displaystyle{\int_{M^n}}\varphi^2 \ e^{-\frac{|X|^2}{2}}dv}, 
\end{equation}
for any function $\varphi$ satisfies $\int_{M^n}\varphi u_j \ e^{-\frac{|X|^2}{2}}dv$, $0\leq j\leq k$.
Since $X: M^n\to \mathbf {R}^{n+p}$ is a self-shrinker in $\mathbf {R}^{n+p}$, we have
\begin{equation}
H=-X^N.
\end{equation}
Letting $x_A$, $A=1, 2, \cdots, n+p $,  denote components of the position vector $X$, 
we define, for $0\leq i\leq k$,  
\begin{equation}
\varphi_i^A:=x_A u_i-\sum_{j=0}^k a_{ij}^Au_j, \quad a_{ij}^A=\int_{M^n}x_Au_iu_j \ e^{-\frac{|X|^2}{2}}dv.
\end{equation}
By a simple calculation, we obtain
\begin{equation}
\displaystyle{\int_{M^n} u_j\varphi_i^Ae^{-\frac{|X|^2}{2}}dv}=0,\hspace{5mm}i,j=0, 1,\cdots, k.
\end{equation}
From the Rayleigh-Ritz inequality, we have
\begin{equation}
\lambda_{k+1} \leq \frac{\displaystyle{-\int_{M^n}}\varphi_i^A\mathfrak L \varphi_i^A  \ e^{-\frac{|X|^2}{2}}dv}
{\displaystyle{\int_{M^n}}(\varphi_i^A)^2 \ e^{-\frac{|X|^2}{2}}dv}.
\end{equation}
Since
\begin{equation}
\begin{aligned}
&\mathfrak L \varphi_i^A=\Delta  \varphi_i^A-\langle X, \nabla  \varphi_i^A\rangle\\
&=\Delta (x_A u_i-\sum_{j=0}^k a_{ij}^Au_j)-\langle X, \nabla  (x_A u_i-\sum_{j=0}^k a_{ij}^Au_j)\rangle\\
&=x_A\Delta u_i + u_i\Delta x_A +2\nabla x_A\cdot \nabla u_i -\langle X, x_A\nabla  u_i+u_i\nabla x_A\rangle\\
&-\sum_{j=0}^k a_{ij}^A\Delta u_j +\langle X, \sum_{j=0}^k a_{ij}^A\nabla u_j)\\
&=-\lambda_ix_Au_i+u_i\mathfrak Lx_A+2\nabla x_A\cdot \nabla u_i+\sum_{j=0}^k a_{ij}^A\lambda_ju_j,
\end{aligned}
\end{equation}
we have, from (3.7) and (3.8),
\begin{equation}
\begin{split}
&(\lambda_{k+1}-\lambda_{i})||\varphi_i^A||^2 \\
&\leq -\int_{M^n}\varphi_i^A(u_i\mathfrak Lx_A+2\nabla x_A\cdot \nabla u_i) \ e^{-\frac{|X|^2}{2}}dv:=W_i^A,
\end{split}
\end{equation}
where 
$$
||\varphi_i^A||^2 =\int_{M^n}(\varphi_i^A)^2 \ e^{-\frac{|X|^2}{2}}dv.
$$
$$
$$
On the other hand, defining
\begin{equation*}
b_{ij}^A=-\int_{M^n}(u_j\mathfrak Lx_A+2\nabla x_A\cdot \nabla u_j)u_i \ e^{-\frac{|X|^2}{2}}dv
\end{equation*} 
we obtain
\begin{equation}
b_{ij}^A=(\lambda_i-\lambda_j)a_{ij}^A.
\end{equation} 
In fact,
\begin{equation*}
\begin{aligned}
\lambda_ia_{ij}^A&=\int_{M^n}\lambda_iu_iu_jx_A e^{-\frac{|X|^2}{2}}dv\\
&=-\int_{M^n}u_jx_A\mathfrak Lu_i \ e^{-\frac{|X|^2}{2}}dv\\
&=-\int_{M^n}u_i\mathfrak L(u_jx_A) \ e^{-\frac{|X|^2}{2}}dv\\
&=-\int_{M^n}u_i(x_A\mathfrak Lu_j+u_j\mathfrak L x_A+2\nabla x_A\cdot \nabla u_j) \ e^{-\frac{|X|^2}{2}}dv\\
&=\lambda_ja_{ij}^A+b_{ij}^A,
\end{aligned}
\end{equation*}
that is, 
\begin{equation*}
b_{ij}^A=(\lambda_i-\lambda_j)a_{ij}^A.
\end{equation*} 
Hence, we have 
\begin{equation}
b_{ij}^A=-b_{ji}^A.
\end{equation}
From (3.6), (3.9)  and the Cauchy-Schwarz  inequality, we infer 
\begin{equation}
\begin{aligned}
W_i^A&=-\int_{M^n}\varphi_i^A\bigl(u_i\mathfrak Lx_A+2\nabla x_A\cdot \nabla u_i\bigl) \ e^{-\frac{|X|^2}{2}}dv\\
&=-\int_{M^n}\varphi_i^A\bigl(u_i\mathfrak Lx_A+2\nabla x_A\cdot \nabla u_i
-\sum_{j=0}^kb_{ij}^Au_j\bigl) \ e^{-\frac{|X|^2}{2}}dv\\
&\leq \|\varphi_i^A\|\|u_i\mathfrak Lx_A+2\nabla x_A\cdot \nabla u_i
-\sum_{j=0}^kb_{ij}^Au_j\|.
\end{aligned}
\end{equation}
Hence, we have, from (3.9) and (3.12),
\begin{equation*}
\begin{split}
&(\lambda_{k+1}-\lambda_{i})(W_i^A)^2\\
&=(\lambda_{k+1}-\lambda_{i})\|\varphi_i^A\|^2\|u_i\mathfrak Lx_A+2\nabla x_A\cdot \nabla u_i
-\sum_{j=0}^kb_{ij}^Au_j\|^2\\
&\leq W_i^A\|u_i\mathfrak Lx_A+2\nabla x_A\cdot \nabla u_i
-\sum_{j=0}^kb_{ij}^Au_j\|^2.\\
\end{split}
\end{equation*} 
Therefore, we obtain
\begin{equation}
(\lambda_{k+1}-\lambda_{i})^2W_i^A
\leq (\lambda_{k+1}-\lambda_{i})\|u_i\mathfrak Lx_A+2\nabla x_A\cdot \nabla u_i
-\sum_{j=0}^kb_{ij}^Au_j\|^2.
\end{equation}
Summing  on $i$ from $0$ to $k$ for (3.13),  we have
\begin{equation}
\sum_{i=0}^k(\lambda_{k+1}-\lambda_{i})^2W_i^A
\leq \sum_{i=0}^k(\lambda_{k+1}-\lambda_{i})\|u_i\mathfrak Lx_A+2\nabla x_A\cdot \nabla u_i
-\sum_{j=0}^kb_{ij}^Au_j\|^2.
\end{equation}
By the definition of $b_{ij}^A$ and (3.10), we have 
\begin{equation}
\begin{aligned}
&\|u_i\mathfrak Lx_A+2\nabla x_A\cdot \nabla u_i
-\sum_{j=0}^kb_{ij}^Au_j\|^2\\
&=\|u_i\mathfrak Lx_A+2\nabla x_A\cdot \nabla u_i\|^2\\
&-2\sum_{j=0}^kb_{ij}^A\int_{M^n}(u_i\mathfrak Lx_A+2\nabla x_A\cdot \nabla u_i)u_j\ e^{-\frac{|X|^2}{2}}dv
+\sum_{j=0}^k(b_{ij}^A)^2\\
&=\|u_i\mathfrak Lx_A+2\nabla x_A\cdot \nabla u_i\|^2-\sum_{j=0}^k(b_{ij}^A)^2\\
&=\|u_i\mathfrak Lx_A+2\nabla x_A\cdot \nabla u_i\|^2-\sum_{j=0}^k (\lambda_i-\lambda_j)^2(a_{ij}^A)^2.\\
\end{aligned}
\end{equation}
Furthermore, according to  the definitions of $W_i^A$ and $\varphi_i^A$, we have from (3.10)
\begin{equation}
\begin{aligned}
W_i^A&=-\int_{M^n}\varphi_i^A\bigl(u_i\mathfrak Lx_A+2\nabla x_A\cdot \nabla u_i\bigl) \ e^{-\frac{|X|^2}{2}}dv\\
&=-\int_{M^n}(x_A u_i-\sum_{j=0}^k a_{ij}^Au_j)\bigl(u_i\mathfrak Lx_A
+2\nabla x_A\cdot \nabla u_i\bigl) \ e^{-\frac{|X|^2}{2}}dv\\
&=-\int_{M^n}(x_A u_i^2\mathfrak Lx_A
+2x_Au_i\nabla x_A\cdot \nabla u_i\bigl) \ e^{-\frac{|X|^2}{2}}dv\\
&+\sum_{j=0}^k a_{ij}^A\int_{M^n}u_j( u_i\mathfrak Lx_A
+2\nabla x_A\cdot \nabla u_i\bigl) \ e^{-\frac{|X|^2}{2}}dv\\
\end{aligned}
\end{equation}
\begin{equation*}
\begin{aligned}
&=-\int_{M^n}\bigl(x_A\mathfrak Lx_A-\frac12\mathfrak L(x_A)^2\bigl)u_i^2 \ e^{-\frac{|X|^2}{2}}dv
+\sum_{j=0}^k a_{ij}^Ab_{ij}^A\\
&=\int_{M^n}\nabla x_A\cdot \nabla x_A  u_i^2\ e^{-\frac{|X|^2}{2}}dv
+\sum_{j=0}^k (\lambda_i-\lambda_j)(a_{ij}^A)^2.\\
\end{aligned}
\end{equation*}
Since 
\begin{equation}
\begin{aligned}
&2\sum_{i,j=0}^k (\lambda_{k+1}-\lambda_{i})^2 (\lambda_i-\lambda_j)(a_{ij}^A)^2\\
&=\sum_{i,j=0}^k (\lambda_{k+1}-\lambda_{i})^2 (\lambda_i-\lambda_j)(a_{ij}^A)^2-\sum_{i,j=0}^k (\lambda_{k+1}-\lambda_{j})^2 (\lambda_i-\lambda_j)(a_{ij}^A)^2\\
\end{aligned}
\end{equation}
\begin{equation*}
\begin{aligned}
&=-\sum_{i,j=0}^k (\lambda_{k+1}-\lambda_{i}+\lambda_{k+1}-\lambda_j) (\lambda_i-\lambda_j)^2(a_{ij}^A)^2\\
&=-2\sum_{i,j=0}^k (\lambda_{k+1}-\lambda_{i}) (\lambda_i-\lambda_j)^2(a_{ij}^A)^2,
\end{aligned}
\end{equation*}
from (3.14), (3.15), (3.16) and (3.17), we obtain, for any $A$, $A=1, 2, \cdots, n+p$,
\begin{equation}
\begin{aligned}
&\sum_{i=0}^k (\lambda_{k+1}-\lambda_{i})^2\int_{M^n}\nabla x_A\cdot \nabla x_A  u_i^2\ e^{-\frac{|X|^2}{2}}dv\\
&\leq \sum_{i=0}^k (\lambda_{k+1}-\lambda_{i})\|u_i\mathfrak Lx_A+2\nabla x_A\cdot \nabla u_i\|^2.
\end{aligned}
\end{equation}
On the other hand, since 
\begin{equation*}
\mathfrak L x_A=-x_A, \quad \sum_{A=1}^{n+p}(\nabla x_A\cdot \nabla u_i)^2=\nabla u_i\cdot\nabla u_i,
\end{equation*}
we infer, from  (2.6), 
\begin{equation}
\begin{aligned}
&\sum_{A=1}^{n+p} \|u_i\mathfrak Lx_A+2\nabla x_A\cdot \nabla u_i\|^2\\
&=\sum_{A=1}^{n+p}\int_{M^n}\bigl(u_i\mathfrak Lx_A+2\nabla x_A\cdot \nabla u_i\bigl)^2\ e^{-\frac{|X|^2}{2}}dv\\
&=\sum_{A=1}^{n+p}\int_{M^n}\bigl(u_i^2(x_A)^2-4u_ix_A\nabla x_A\cdot \nabla u_i
+4(\nabla x_A\cdot \nabla u_i)^2\bigl)\ e^{-\frac{|X|^2}{2}}dv\\
&=\sum_{A=1}^{n+p}\int_{M^n}\bigl(u_i^2(x_A)^2-\nabla (x_A)^2\cdot \nabla u_i^2\bigl)\ e^{-\frac{|X|^2}{2}}dv
+4\int_{M^n}\nabla u_i\cdot\nabla u_i\ e^{-\frac{|X|^2}{2}}dv\\
\end{aligned}
\end{equation}
\begin{equation*}
\begin{aligned}
&=\int_{M^n}(\mathfrak L|X|^2+|X|^2) u_i^2\ e^{-\frac{|X|^2}{2}}dv+4\lambda_i\\
&=\int_{M^n}(2n-|X|^2) u_i^2\ e^{-\frac{|X|^2}{2}}dv+4\lambda_i\\
&\leq(2n-\min_{M^n}{|X|^2})+4\lambda_i.\\
\end{aligned}
\end{equation*}
Furthermore, because of 
\begin{equation}
\sum_{A=1}^{n+p}\nabla x_A\cdot \nabla x_A =n,
\end{equation}
taking summation on $A$ from $1$ to $n+p$ for (3.18) and using  (3.19) and (3.20), we get 
\begin{equation*}
\begin{aligned}
&\sum_{i=0}^k (\lambda_{k+1}-\lambda_{i})^2\leq \frac 4n\sum_{i=0}^k (\lambda_{k+1}-\lambda_{i})
(\lambda_i+\frac {2n-\min_{M^n}{|X|^2}}4).
\end{aligned}
\end{equation*}
It finished the proof of the theorem 1.1.
\begin{flushright}
$\square$
\end{flushright}

\section{Upper bounds for eigenvalues}

\noindent
The following recursion formula of Cheng and Yang \cite{CY} plays a very important role in order to
prove the theorem 1.2.

\vskip 2mm
\noindent
{\bf A recursion formula of Cheng and Yang}.  Let  $\mu_1 \leq  \mu_2 \leq  \dots,
\leq \mu_{k+1}$ be any positive  real numbers satisfying
\begin{equation*}
  \sum_{i=1}^k(\mu_{k+1}-\mu_i)^2 \le
 \frac 4n\sum_{i=1}^k\mu_i(\mu_{k+1} -\mu_i).
\end{equation*}
 Define
 \begin{equation*}
 \Lambda_k=\frac 1k\sum_{i=1}^k\mu_i,\qquad T_k=\frac 1k
\sum_{i=1}^k\mu_i^2, \ \ \
F_k=\left (1+\frac 2n \right )\Lambda_k^2-T_k.
\end{equation*}
Then, we have
\begin{equation}
F_{k+1}\leq C(n,k) \left ( \frac {k+1}k \right )^{\frac 4n}F_k,
\end{equation}
where
$$
C(n,k) =1-\frac 1{3n}
  \left (\frac k{k+1}\right )^{\frac
  4n}\frac {\left(1+\frac 2n\right )\left (1+
  \frac 4n\right )}{(k+1)^3}<1.
$$

\vskip 2mm
\noindent
{\it Proof of Theorem {\rm 1.2}.}
From the proposition 2.1, we know 
$$
\mu_{i+1}=\lambda_{i}+\dfrac{2n-\min_{M^n} |X|^2}{4}>0,
$$
for any $i=0, 1, 2, \cdots$.  Then, we obtain from (1.5)
\begin{equation}
\sum_{i=1}^k (\mu_{k+1}-\mu_{i})^2 \leq \frac{4}{n}\sum_{i=1}^k
(\mu_{k+1}-\mu_{i})\mu_i.
\end{equation} 
Thus, we know that $\mu_i$'s satisfy the condition of the above recursion formula of Cheng and Yang \cite{CY}.
Furthermore,  since
\begin{equation*}
\mathfrak L x_A=-x_A  \ \text{\rm and} \   \int_{M^n}x_A\ e^{-\frac{|X|^2}{2}}dv=0,  \quad \text{\rm for} \  A=1, 2, \cdots, n+p,
\end{equation*}
$\lambda=1$ is an eigenvalue of $\mathfrak L$ with multiplicity at least $n+p$.
Thus, 
$$
\lambda_{1}\leq \lambda_{2}\leq \cdots\leq \lambda_{n+1}\leq 1.
$$
Hence,  we have 
\begin{equation}
\sum_{j=1}^{n}(\mu_{j+1}-\mu_1)=\sum_{j=1}^{n}\lambda_{j}\leq n\leq 2n-\min_{M^n} |X|^2=4\mu_1
\end{equation}
because of $\min_{M^n} |X|^2\leq n$ according to the proposition 2.1.
Hence, we can prove  the theorem 1.2  as in Cheng and Yang \cite{CY} almost word by word. 
For the convenience of readers, we shall give a self contained proof.
First of all, according to the  above recursion formula of Cheng and Yang, we have
\begin{equation*}
F_{k} \le  C(n,k-1) \left (\frac k{k-1} \right )^ {\frac
4n}F_{k-1} \le  k^{\frac 4n} F_1=\frac 2n k^{\frac 4n}\mu_1^2.
\end{equation*}
Furthermore, we infer,  from (4.2) 
$$
\left [\mu_{k+1} -\left (1+\frac 2n\right)
 \Lambda_k \right ]^2
\le \left (1+\frac 4n\right )F_k-\frac 2n \left (1+\frac 2n\right
)\Lambda_k^2.
$$
Hence, we have
$$
\frac {\frac 2n} {\left (1+\frac 4n\right)}\mu_{k+1}^2+ \frac
{1+\frac2n}{1+\frac 4n } \left (\mu_{k+1}- \left (1+\frac
4n\right )\Lambda_k \right )^2 \le \left (1+\frac 4n\right )F_k.
$$
Thus, we derive
\begin{equation}
\mu_{k+1} \le \left (1+\frac 4n\right ) \sqrt{\frac n2 F_k}
 \le \left (1+\frac 4n\right ) k^{\frac 2n}\mu_1.
\end{equation}
Define 
\begin{align*}
&a_1(n)=\frac{n(1+\frac4n)\left(1+\frac 8{n+1}+\frac 8{(n+1)^2}
\right )^{\frac 12}}{(n+1)^{\frac 2n}}-n,\\
&a_2(k,n)=\frac n{k^{\frac 2n}}\left (1+\frac
{4(n+k+4)}{n^2+5n-4(k-1)}\right)-n,\\
&a_2(k)=\max \{ a(n,k),k \le n\le 400\}, \\
 &a_3(k)=\dfrac{4}{1-\frac k{400}}-2\log k,\\
 &a(k) =\max \{a_1(k), a_2(k+1)), a_3(k+1) \}.
\end{align*}

\vskip2mm
\noindent
{\it The case 1}.  For   $k \ge n+1$, we have 
\begin{equation}
\aligned \mu_{k+1}\le &\frac {\left (1+\frac 4n\right)
\left(1+\frac 8{n+1}+\frac 8{(n+1)^2}
\right )^{\frac 12}}{(n+1)^{\frac 2n}}k^{\frac 2n}\mu_1\\
&=\left (1+\frac {a_1(n)}n\right)k^{\frac 2n}\mu_1,
\endaligned
\end{equation}
where $a_1(n) \le 2.31$. In fact, 
since 
$\mu_{k+1}$ satisfies (4.2),  we have, from (4.1),
\begin{equation}
\mu_{k+1}^2 \le  \frac n2 \left (1+\frac 4n\right)^2F_k \le
\frac n2 \left (1+\frac 4n\right)^2 \left(\frac k{n+1}\right
)^{\frac 4n}F_{n+1}.
\end{equation}
On the other hand,
\begin{equation}
\aligned F_{n+1}=&\frac 2n\Lambda_{n+1}^2-\sum_{i=1}^{n+1}\frac
{(\mu_i-\Lambda_{n+1})^2}{n+1}\\
\le&\frac 2n\Lambda_{n+1}^2-\frac {(\mu_1-\Lambda_{n+1})^2+
\frac 1n(\mu_1-\Lambda_{n+1})^2}{n+1}\\
=&\frac 2n\left (\Lambda_{n+1}^2-\frac
{(\mu_1-\Lambda_{n+1})^2}2\right).
\endaligned
\end{equation}
It is obvious that $\Lambda_{n+1}^2-\dfrac
{(\mu_1-\Lambda_{n+1})^2}2$ is an increasing function of
$\Lambda_{n+1}$.  From (4.3), we have
\begin{equation}
\mu_{n+1}+\dots +\mu_2 \le (n+4)\mu_1.
\end{equation}
Thus,  we derive
\begin{equation}
\Lambda_{n+1} \le (1+\frac4{n+1})\mu_1.
\end{equation}
Hence, we have
\begin{equation}
\frac n2F_{n+1}\leq \left(1+\frac 8{n+1}+\frac 8{(n+1)^2}
\right)\mu_1^2.
\end{equation}
From (4.6) and (4.10), we complete the proof of  (4.5).
\vskip2mm
\noindent
{\it The case 2}.   For $k \ge 55$ and $n\ge 54$, we have
\begin{equation}
\aligned \mu_{k+1}\le k^{\frac 2n}\mu_1.
\endaligned
\end{equation}

\noindent
If $ k\ge n+1$, from the case 1, we have
$$
\mu_{k+1} \le\frac 1{(n+1)^{\frac 2n}} \left (1+ \frac
4n\right )^2 k^{\frac 2n} \mu_1.
$$
Since
\begin{equation}
\begin{aligned} &(n+1)^{\frac 2n}=\exp \left (\frac 2n\log (n+1)\right)\\
&\ge 1+\frac 2n\log (n+1)+\frac 2{n^2}(\log (n+1))^2\\
&\ge \left (1+\frac 1n\log (n+1)\right )^2,
\end{aligned}
\end{equation}
we have
\begin{equation}
\mu_{k+1} \le \left (\frac{ 1+\frac 4n}{1+\frac 1n \log
(n+1)}\right )^2 k^{\frac 2n}\mu_1.
\end{equation}
Then, when $n\ge 54$, $\log (n+1) \geq  4 $,  we have
$$
\mu_{k+1} \le k^{\frac 2n}\mu_1.
$$
On the other hand, if $k\le n$, then $ \Lambda_k \le
\Lambda_{n+1}$.
Since
\begin{align*}
\frac n2F_k&
=\Lambda_k^2-\frac{n}2\frac{\sum_{i=1}^k(\mu_i-\Lambda_k)^2}k\\
  &\leq\Lambda_k^2-\frac{n}2\frac{(\mu_1-\Lambda_k)^2
  +\dfrac{\left\{\sum_{i=2}^k(\mu_i-\Lambda_k)\right\}^2}{k-1}}k\\
&\leq\Lambda_{k}^2-\frac{(\mu_1-\Lambda_{k})^2}2\\
&\leq
\Lambda_{n+1}^2-\frac{(\mu_1-\Lambda_{n+1})^2}2\leq
(1+\frac4n)^2\mu_1^2,
\end{align*}
 we have
$$
\aligned \mu_{k+1} \le &\left (1+\frac 4n \right ) \sqrt{\frac
n2 F_k} \le \frac 1{k^{\frac 2n}}\left ( 1+\frac 4n\right )^2
      k^{\frac 2n}\mu_1      \le \left (\frac {1+\frac 4 n}{1+\frac {\log k}n}
      \right )^2
      k^{\frac 2n} \mu_1.
      \endaligned
      $$
 Here we used $ k^{\frac 2n} \geq(1+\frac {\log k}n)^2$.
By the same assertion as above, when $k\geq 55$, we also have
$$
\mu_{k+1} \leq k^{\frac 2n}\mu_1.
$$
\vskip2mm
\noindent
 {\it The case 3}.  For $k\leq 54$ and $k\leq n$, we have 
 \begin{equation*}
\begin{aligned}
 \mu_{k+1}\leq (1+\frac{\max\{a_2(k),a_3(k)\}}n)k^{\frac2n}\mu_1.
\end{aligned}
\end{equation*}
Because of $k\le n$ and $k \le 54$,
from (4.3), we derive,
\begin{equation}
\mu_{k+1} \le \frac{1}{n-k+1}\{(n+5)\mu_1 -k\Lambda_k\}.
\end{equation}
Since the formula (4.2) is a quadratic inequality for $\mu_{k+1}$, we have
 \begin{equation}
  \mu_{k+1} \le \left (1+\frac 4n \right )\Lambda_k.
\end{equation}
Since the right hand side of (4.14) is a decreasing function of
$\Lambda_k$ and the right hand side of (4.15) is an increasing
function of $\Lambda_k$, for $\dfrac{1}{n-k+1}\{(n+5)\mu_1
-k\Lambda_k\} =\left (1+\frac 4n \right )\Lambda_k$,  we infer
\begin{equation}
\begin{aligned} \mu_{k+1}\le&\frac 1{k^{\frac 2n}}\left (1+\frac
{4(n+k+4)}{n^2+5n-4(k-1)}\right)
k^{\frac 2n}\mu_1\\
&=\left (1+\frac { a_2(k,n)}n\right )k^{\frac 2n}\mu_1.
\end{aligned}
\end{equation}
From the definition of  $a_2(k)=\max \{ a(n,k),k \le n\le 400\}$, when $n\leq 400$, we obtain
\begin{equation}
\mu_{k+1} \le \left (1+\frac { a_2(k)}n\right )k^{\frac2n}\mu_1.
\end{equation}
When $n > 400$ holds, from (4.4), we have
$$
\mu_{k+1}\le \left ( 1+\frac 4{n-k}\right )\mu_1.
$$
Since $n>400$ and $k\leq 54$, we know $\frac2n\log k<\frac1{50}$.
Hence, we have
\begin{align*}
k^{-\frac2n}=e^{-\frac2n\log k}
&=1-\frac2n\log k+\frac12(\frac2n\log k)^2-\cdots\\
&\leq 1-\frac2n\log k+\frac12(\frac2n\log k)^2.
\end{align*}
Therefore, we obtain
\begin{align*}
&(1+\frac{4}{n-k})k^{-\frac2n}\\
&\leq(1+\frac{4}{n-k})\left(1-\frac2n\log k+\frac12(\frac2n\log k)^2\right)\\
&\leq 1+\frac {\left (4/(1-\frac k{400})-2\log k \right)}n.
\end{align*}
Hence, we infer
\begin{equation}
\begin{aligned}
 \mu_{k+1}\le& \left ( 1+\frac 4{n-k}\right )k^{-\frac2n}k^{\frac2n}\mu_1\\
\le& \left (1+\frac {\left (4/(1-\frac k{400})-2\log k \right)}
n\right )k^{\frac 2n}\mu_1\\
=&\left (1+ \frac {a_3(k)}n\right )k^{\frac 2n}\mu_1.
\end{aligned}
\end{equation}
By Table 1 of the values of $a_1(k)$, $a_2(k+1)$ and $a_3(k+1)$
which are calculated by using Mathematica and are listed up in the
next page,  we have
 $ a_1(1) \le a_2(2) \leq a_3(2) =a(1)\le 2.64$ and, for $ k\geq 2$,
$$
a_3(k+1) \le a_2 (k+1) \le a_1(k).
$$
Hence, $a(k)=a_1(k)$ for $k\geq 2$.
Further, for $k\geq 41$, we know $a(k)<0$.
Hence, for $k\geq 2$, we derive
$$
\mu_{k+1}\leq (1+\frac{a(\min\{n,k-1\})}n)k^{\frac2n}\mu_1
$$
and for $n\geq 41$ and $k\geq 41$,  we have
$$
\mu_{k+1}\leq k^{\frac2n}\mu_1.
$$
When $k=1$, $a(0)=4$ from (4.4). It is easy to check that, when $k\ge 3$,
 by a simple calculation,
$$
a(k) \le 2.2-4\log (1+\frac{k-3}{50}).
$$
This completes the proof of the theorem 1.2.
\begin{flushright}
$\square$
\end{flushright} 

\vskip 3mm
{\begin{center}{Table 1: The values of $a_1(k)$, $a_2(k+1)$ and $a_3(k+1)$}
\end{center}}
\begin{table}[htbp]
\begin{center}
\begin{tabular}{|c|rrrrrrrrrrrr} \hline
$k$                &1&2&3&4&5&6&7&8&9&10&\\ \hline
$a_1(k)\leq $  &2.31&2.27&2.2&2.12&2.03&1.94&1.86&1.77&1.69&1.61\\ \hline
$a_2(k+1)\leq $   &2.62&2.05&2.00&1.96&1.90&1.84&1.77&1.70&1.63&1.56\\ \hline
$a_3(k+1)\leq $  &2.64&1.84&1.27&0.84&0.48&0.18&-0.07&-0.30&-0.50&-0.68\\
 \hline
 \\
$k$                &11&12&13&14&15&16&17&18&19&20&\\ \hline
$a_1(k)\leq$  &1.53&1.46&1.39&1.32&1.25&1.18&1.12&1.06&1.00&0.94\\ \hline
$a_2(k+1)\leq $   &1.49&1.42&1.35&1.29&1.22&1.16&1.10&1.04&0.98&0.92\\ \hline
$a_3(k+1) \leq $  &-0.84&-0.99&-1.13&-1.26&-1.37&-1.48&-1.59&-1.68&-1.78&-1.86\\ \hline
 \\
$k$                &21&22&23&24&25&26&27&28&29&30&\\ \hline
$a_1(k)\leq $  &0.89&0.83&0.78&0.72&0.67&0.62&0.58&0.53&0.48&0.44\\ \hline
$a_2(k+1)\leq $   &0.87&0.82&0.76&0.71&0.66&0.61&0.57&0.52&0.47&0.43&\\ \hline
$a_3(k+1)\leq $  &-1.94&-2.02&-2.10&-2.17&-2.23&-2.30&-2.36&-2.42&-2.47&-2.53\\ \hline
 \\
 $k$                &31&32&33&34&35&36&37&38&39&40&41&\\ \hline
$a_1(k)\leq $  &0.39&0.35&0.31&0.27&0.23&0.19&0.15&0.11&0.07&0.03&-0.00\\ \hline
$a_2(k+1)\leq $  &0.38&0.34&0.30&0.26&0.22&0.18&0.14&0.10&0.07&0.03&-0.01\\ \hline
$a_3(k+1)\leq $  &-2.58&-2.63&-2.68&-2.72&-2.77&-2.81&-2.85&-2.89&-2.93&-2.97&-3.00\\ \hline
\end{tabular}
\end{center}
\end{table}

\vskip 2mm
\noindent

\section{The Dirichlet eigenvalue problem}

\noindent
For a bounded domain $\Omega$  with a piecewise smooth boundary $\partial \Omega$ in  an $n$-dimensional  complete
self-shrinker  in $\mathbf{R}^{n+p}$, we consider 
the following Dirichlet eigenvalue problem of  the differential operator $\mathfrak L$:
\begin{equation}
\begin{cases}
\mathfrak Lu=-\lambda u  &  \text{in $\Omega$}, \\
u=0 & \text{on $\partial \Omega$}.
\end{cases}
\end{equation} 
This eigenvalue problem has a real and  discrete spectrum: 
\begin{equation*}
0<\lambda _{1} < \lambda _{2} \leq \cdots \leq \lambda _{k} \leq \cdots 
 \longrightarrow \infty,
\end{equation*} 
where each eigenvalue is repeated according to  its multiplicity.  We have  following 
estimates for eigenvalues of the  Dirichlet eigenvalue problem $(5.1)$.

\begin{theorem} Let $\Omega$ be a bounded domain with a piecewise smooth boundary $\partial \Omega$ 
in an $n$-dimensional complete 
self-shrinker  $M^n$   in $\mathbf{R}^{n+p}$. Then, 
eigenvalues of the  Dirichlet eigenvalue problem $(5.1)$ satisfy
\begin{equation*}
\sum_{i=1}^k (\lambda_{k+1}-\lambda_{i})^2 \leq \frac{4}{n}\sum_{i=1}^k
(\lambda_{k+1}-\lambda_{i})(\lambda_{i}+\frac{2n-\inf_{\Omega}{|X|^2}}{4}).
\end{equation*} 
\end{theorem}
\noindent
\begin{proof}. By making use of the same proof as in the proof of the theorem 1.1,
we can prove the theorem 5.1  if one notices to count the number of eigenvalues from $1$.
\end{proof}

\noindent
From the recursion formula of Cheng and Yang \cite{CY}, we can give an upper bound for eigenvalue $\lambda_{k+1}$:
\begin{theorem} Let $\Omega$ be a bounded domain with a piecewise smooth boundary $\partial \Omega$
in an $n$-dimensional complete 
self-shrinker  $M^n$   in $\mathbf{R}^{n+p}$. Then, 
eigenvalues of the  Dirichlet eigenvalue problem $(5.1)$ satisfy, for any $k\geq 1$, 
\begin{equation*}
\lambda_{k+1}+\frac{2n-\inf_{\Omega}{|X|^2}}{4}
\leq (1+\frac {a(min\{n,k-1\})}{n})(\lambda_1+\frac{2n-\inf_{\Omega}{|X|^2}}{4})k^{2/n},
\end{equation*}
where the bound of $a(m)$ can be formulated as:
$$
\left \{ \aligned
a(0)&\leq 4, \\
 a(1)&\leq 2.64,\\
 a(m)&\leq 2.2-4\log(1+\frac1{50}(m-3)),\qquad \mbox{for}\quad
 m\geq 2.
\endaligned \right .
$$
In particular, for $n\ge 41$ and $k\geq 41$, we have
 \begin{equation*}
\lambda_{k+1}+\frac{2n-\inf_{\Omega}{|X|^2}}{4}
\leq (\lambda_1+\frac{2n-\inf_{\Omega}{|X|^2}}{4})k^{2/n}.
\end{equation*}
\end{theorem}

\begin{remark}
For the Euclidean space $\mathbf{R}^{n}$, the differential operator $\mathfrak L$ is 
called Ornstein-Uhlenbeck operator in stochastic analysis. Since the Euclidean
space $\mathbf{R}^{n}$ is a complete self-shrinker in $\mathbf{R}^{n+1}$, our 
theorems also give estimates for eigenvalues of the Dirichlet
eigenvalue problem of the Ornstein-Uhlenbeck operator.
\end{remark}

\noindent
For the  Dirichlet eigenvalue problem $(5.1)$, components $x_A$'s of the position vector $X$  are not eigenfunctions
corresponding to the eigenvalue $1$ because they do not satisfy the boundary condition.
In order to prove the theorem 5.2, we need to obtain the following  estimates for lower order eigenvalues. 

\begin{proposition} Let $\Omega$ be a bounded domain with a piecewise smooth boundary $\partial \Omega$
 in an $n$-dimensional complete 
self-shrinker  $M^n$   in $\mathbf{R}^{n+p}$. Then, 
eigenvalues of the  Dirichlet eigenvalue problem $(5.1)$ satisfy
\begin{equation*}
\sum_{j=1}^n (\lambda_{j+1}-\lambda_{1}) \leq (2n-\inf_{\Omega}{|X|^2})+4\lambda_1.
\end{equation*} 
\end{proposition}
\noindent

\begin{proof}
Let  $u_j$ be an  eigenfunction corresponding
to the eigenvalue $\lambda_j$ such that 
\begin{equation}
\begin{cases}
\begin{aligned}
&\mathfrak L u_j=- \lambda_j u_j \quad \text{in} \ {\Omega}\\
&u_j=0,  \ \ \text{on $\partial \Omega$}\\
&\int_{\Omega}u_iu_j \ e^{-\frac{|X|^2}{2}}dv=\delta_{ij}, \ \text{for any} \ i, j=1, 2, \cdots. 
\end{aligned}
\end{cases}
\end{equation}
We consider  an  $(n+p)\times (n+p)$-matrix $B=(b_{AB})$ defined
by 
$$
b_{AB}=\int_{\Omega} x_{A} u_1u_{B+1}e^{-\frac{|X|^2}{2}}dv.
$$
From the orthogonalization of Gram and Schmidt, there exist an upper triangle matrix $R=(R_{AB})$ and  an orthogonal matrix $Q=(q_{AB})$ such that
$R=QB$. Thus, 
\begin{equation}
R_{AB}=\sum_{C=1}^{n+p}q_{AC}b_{CB}
=\int_{\Omega} \sum_{C=1}^{n+p}q_{AC}x_C u_1u_{B+1}=0, \ \text{for} \ 1\leq B <A\leq n+p.
\end{equation}
Defining $y_A=\sum_{C=1}^{n+p}q_{AC}x_C$, we have 
\begin{equation}
\int_{\Omega}y_A u_1u_{B+1}=\int_{\Omega}\sum_{C=1}^{n+p}q_{AC}x_C u_1u_{B+1}=0, \ \text{for} \ 1\leq B <A\leq n+p.
\end{equation}
Therefore,  the  functions $\varphi_{A}$ defined by 
$$
\varphi_{A}=(y_A-a_A)u_1, \  a_A=\int_{\Omega} y_Au_1^2 \ e^{-\frac{|X|^2}{2}}dv,
\  \text{for}   \ \ 1\leq A \leq n+p
$$
satisfy
$$
\int_{\Omega}\varphi_{A}u_{B}=0,\qquad \mbox{for}\  1\leq B \leq A \leq n+p.
$$
Therefore, $\varphi_A$ is a trial function.
From the Rayleigh-Ritz inequality, we have, for $1\leq A\leq n+p$, 
\begin{equation}
\lambda_{A+1} \leq \frac{\displaystyle{-\int_{\Omega}}\varphi_A\mathfrak L \varphi_A  \ e^{-\frac{|X|^2}{2}}dv}
{\displaystyle{\int_{\Omega}}(\varphi_A)^2 \ e^{-\frac{|X|^2}{2}}dv}.
\end{equation}
From the definition of $\varphi_{A}$, we derive
\begin{equation*}
\begin{aligned}
&\mathfrak L \varphi_A=\Delta  \varphi_A-\langle X, \nabla  \varphi_A\rangle\\
&=\Delta \{(y_A-a_A) u_1\}-\langle X, \nabla   \{(y_A-a_A) u_1\}\rangle\\
&=y_A\mathfrak Lu_1 + u_1\mathfrak Ly_A +2\nabla y_A\cdot \nabla u_1 -a_A\mathfrak Lu_1\\
&=-\lambda_1y_Au_1-u_1y_A+2\nabla y_A\cdot \nabla u_1+a_A\lambda_1u_1.
\end{aligned}
\end{equation*}
Thus, (5.5) can be written as
\begin{equation}
(\lambda_{A+1}-\lambda_1)\|\varphi_{A}\|^2
\leq \int_{\Omega} \bigl(y_A u_1-2\nabla y_A\cdot\nabla u_1\bigl)\varphi_A \ e^{-\frac{|X|^2}{2}}dv.
\end{equation}
From the Cauchy-Schwarz inequality, we obtain
$$
\left(\int_{\Omega} \bigl(y_A u_1-2\nabla y_A\cdot\nabla u_1\bigl)\varphi_A \ e^{-\frac{|X|^2}{2}}dv\right)^2\leq
\|\varphi_{A}\|^2 \|y_Au_1-2\nabla y_A\cdot\nabla u_1\|^2.
$$
Multiplying the above inequality by $(\lambda_{A+1}-\lambda_1)$, we infer, from  (5.6), 
\begin{equation}
\begin{aligned}
&(\lambda_{A+1}-\lambda_1)\left(\int_{\Omega}\bigl(y_A u_1-2\nabla y_A\cdot\nabla u_1\bigl)\varphi_A \ e^{-\frac{|X|^2}{2}}dv\right)^2\\
 &\leq (\lambda_{A+1}-\lambda_1)\|\varphi_{A}\|^2 \|y_A u_1-2\nabla y_A\cdot\nabla u_1\|^2\\
&\leq \Big(\int_{\Omega} \bigl(y_A u_1-2\nabla y_A\cdot\nabla u_1\bigl)\varphi_A \ e^{-\frac{|X|^2}{2}}dv\Big)
 \|y_Au_1-2\nabla y_A\cdot\nabla u_1\|^2\\
\end{aligned}
\end{equation}
Hence, we derive 
\begin{equation}
\begin{aligned}
&(\lambda_{A+1}-\lambda_1)\int_{\Omega} 
\bigl(y_A u_1-2\nabla y_A\cdot\nabla u_1\bigl)\varphi_A \ e^{-\frac{|X|^2}{2}}dv
\leq   \|y_A u_1-2\nabla y_A\cdot\nabla u_1\|^2\\
\end{aligned}
\end{equation}
Since 
$$
\sum_{A=1}^{n+p} y_A^2= \sum_{A=1}^{n+p} x_A^2=|X|^2,
$$
we infer
\begin{equation}
\begin{aligned}
&\sum_{A=1}^{n+p} \|y_A u_1-2\nabla y_A\cdot\nabla u_1\|^2\\
&=\sum_{A=1}^{n+p}\int_{\Omega}\bigl(y_A ^2u_1^2
-4y_Au_1\nabla y_A\cdot\nabla u_1+4(\nabla y_A\cdot\nabla u_1)^2\bigl)
 \ e^{-\frac{|X|^2}{2}}dv\\
 &=\int_{\Omega} \bigl(|X|^2u_1^2
-\nabla |X |^2\cdot\nabla u_1^2+4\nabla u_1\cdot\nabla u_1\bigl)\ e^{-\frac{|X|^2}{2}}dv\\
&=\int_{\Omega} \bigl(|X|^2u_1^2+
\mathfrak L |X |^2 u_1^2+4\nabla u_1\cdot\nabla u_1\bigl)\ e^{-\frac{|X|^2}{2}}dv\\
&=\int_{\Omega} (2n-|X|^2)u_1^2\ e^{-\frac{|X|^2}{2}}dv+4\lambda_1\leq (2n-\inf_{\Omega} |X|^2)+4\lambda_1.
\end{aligned}
\end{equation}
\vskip2mm
\noindent
On the other hand, from the definition of $\varphi_A$, we have 
\begin{equation}
\begin{aligned}
&\int_{\Omega} \bigl(y_A u_1-2\nabla y_A\cdot\nabla u_1\bigl)\varphi_A \ e^{-\frac{|X|^2}{2}}dv\\
&=\int_{\Omega}\bigl(y_A^2 u_1^2-a_Ay_Au_1^2
+2a_Au_1\nabla y_A\cdot\nabla u_1-2y_Au_1\nabla y_A\cdot\nabla u_1\bigl)\ e^{-\frac{|X|^2}{2}}dv\\
&=\int_{\Omega} \bigl(y_A^2 u_1^2-a_Ay_Au_1^2
-a_A\mathfrak L y_Au_1^2+\frac{1}{2}\mathfrak L y_A^2 u_1^2\bigl)\ e^{-\frac{|X|^2}{2}}dv\\
&=\int_{\Omega} \bigl(y_A^2 u_1^2+\frac{1}{2}\mathfrak L y_A^2 u_1^2\bigl)\ e^{-\frac{|X|^2}{2}}dv\\
&=\int_{\Omega} \nabla y_A\cdot\nabla y_A u_1^2\ e^{-\frac{|X|^2}{2}}dv.\\
\end{aligned}
\end{equation}
For any point $p$,  we choose
a new coordinate system $\bar X=(\bar x_1,\cdots,\bar x_{n+p})$ of $\mathbf  R^{n+p}$
given by  $X-X(p)=\bar X O$ such that
 $(\frac {\partial }{\partial \bar x_1})_p,\cdots,
 (\frac {\partial }{\partial \bar x_n})_p$ span $T_pM^n$
 and at $p$, $g\Big(\frac{\partial}{\partial \bar x_i},\frac{\partial}{\partial \bar x_j}\Big)
 =\delta_{ij}$, 
 where  $O=(o_{AB})\in O(n+p)$ is an $(n+p)\times (n+p)$ orthogonal matrix.
 \begin{equation*}
 \begin{aligned}
 &\nabla y_A\cdot\nabla y_A=g(\nabla y_A, \nabla y_A)
 = \sum_{B,C=1}^{n+p}q_{AB}q_{AC}g(\nabla x_B, \nabla x_C)\\
&=\sum_{B,C=1}^{n+p}q_{AB}q_{AC}
g(\sum_{D=1}^{n+p}o_{DB}\nabla \bar x_{D}, \sum_{E=1}^{n+p}o_{EC}\nabla \bar x_E)\\
&=\sum_{B,C,D,E=1}^{n+p}q_{AB}o_{DB}q_{AC}o_{EC}
g(\nabla \bar x_{D},\nabla \bar x_E)\\
&=\sum_{j=1}^n\bigl(\sum_{B=1}^{n+p}q_{AB}o_{jB}\bigl)^2\leq 1
\end{aligned}
\end{equation*}
since $OQ$ is an   orthogonal matrix  if $Q$ and $O$ are orthogonal matrices, 
that is, we have  
 \begin{equation}
 \begin{aligned}
 \nabla y_A\cdot\nabla y_A\leq 1.
\end{aligned}
\end{equation}
Thus, we obtain, from (5.10) and (5.11),
\begin{equation}
\begin{aligned}
&\sum_{A=1}^{n+p}(\lambda_{A+1}-\lambda_1)
\int_{\Omega} \bigl(y_A u_1-2\nabla y_A\cdot\nabla u_1\bigl)\varphi_A \ e^{-\frac{|X|^2}{2}}dv\\
&=\sum_{A=1}^{n+p}(\lambda_{A+1}-\lambda_1)\int_{\Omega}  \nabla y_A\cdot\nabla y_A u_1^2\ e^{-\frac{|X|^2}{2}}dv\\
&=\sum_{j=1}^{n}(\lambda_{j+1}-\lambda_1)\int_{\Omega}  \nabla y_j\cdot\nabla y_j u_1^2\ e^{-\frac{|X|^2}{2}}dv\\
&+\sum_{A=n+1}^{n+p}(\lambda_{A+1}-\lambda_1)\int_{\Omega}  \nabla y_A\cdot\nabla y_Au_1^2\ e^{-\frac{|X|^2}{2}}dv\\
&\geq \sum_{j=1}^{n}(\lambda_{j+1}-\lambda_1)\int_{\Omega}\nabla y_j\cdot\nabla y_j  u_1^2\ e^{-\frac{|X|^2}{2}}dv\\
&+\sum_{A=n+1}^{n+p}(\lambda_{n+1}-\lambda_1)\int_{\Omega} \nabla y_A\cdot\nabla y_Au_1^2\ e^{-\frac{|X|^2}{2}}dv\\
&=\sum_{j=1}^{n}(\lambda_{j+1}-\lambda_1)\int_{\Omega} \nabla y_j\cdot\nabla y_j  u_1^2\ e^{-\frac{|X|^2}{2}}dv\\
&+(\lambda_{n+1}-\lambda_1)\int_{\Omega} (n-\sum_{j=1}^{n}\nabla y_j\cdot\nabla y_j ) u_1^2\ e^{-\frac{|X|^2}{2}}dv\\
\end{aligned}
\end{equation}
\begin{equation*}
\begin{aligned}
&=\sum_{j=1}^{n}(\lambda_{j+1}-\lambda_1)\int_{\Omega} \nabla y_j\cdot\nabla y_j  u_1^2\ e^{-\frac{|X|^2}{2}}dv\\
&+(\lambda_{n+1}-\lambda_1)\int_{\Omega}\sum_{j=1}^{n}(1-\nabla y_j\cdot\nabla y_j ) u_1^2\ e^{-\frac{|X|^2}{2}}dv\\
&\geq \sum_{j=1}^{n}(\lambda_{j+1}-\lambda_1).
\end{aligned}
\end{equation*}
According to (5.8), (5.9) and (5.12), we obtain
$$
\sum_{j=1}^{n}(\lambda_{j+1}-\lambda_1)\leq (2n-\inf_{\Omega} |X|^2)+4\lambda_1.
$$
This completes the proof of the proposition 5.1.
\end{proof}
\vskip 2mm

\noindent
{\it Proof of Theorem {\rm 5.2}}. By making use of the proposition 5.1 and the same proof as in the proof of the theorem 1.2,
we can prove the theorem 5.2  if one notices to count the number of eigenvalues from $1$.

\begin{flushright}
$\square$
\end{flushright}

\end{document}